%% file: article8.tex
\documentclass[11pt,twoside]{article}

\newcommand{\HeadTitle}{Beta Analogues and Lerch Unification}

\newcommand{\HeadTitleTwo}{\begin{center}
\Large{\textit{Dirichlet Beta Analogues, Sign Alternation, and Lerch Unification of Hyperbolic and Logarithmic Tangent Integrals}}
\end{center}}

\input{preamble}
\input{font}

\begin{document}

\input{titlepage}

\begin{abstract}
We develop a Dirichlet-beta counterpart and a structural extension of
the coefficient-extraction framework for hyperbolic and logarithmic
tangent integrals.  For integers $m\geq n\geq1$ with $m+n$ even,
the zeta-type formula
\[
\int_0^\infty\frac{\tanh^{m+1}x}{x^{n+1}}\,dx
=
(-1)^{(m-n)/2}
\sum_{p=\lceil n/2\rceil}^{(m+n)/2}
\binom{2p}{n}(2^{2p+1}-1)
\frac{\zeta(2p+1)}{\pi^{2p}}
[u^{m+n-2p}](u\cot u)^{m+1}
\]
admits the beta-type analogue
\[
\int_0^\infty
\frac{\tanh^m x}{x^n\cosh x}\,dx
=
(-1)^{(m-n)/2}
\sum_{p=\lceil n/2\rceil}^{(m+n)/2}
2^{2p}\binom{2p-1}{n-1}
\frac{\beta(2p)}{\pi^{2p-1}}
[u^{m+n-2p}]
\frac{u}{\sin u}(u\cot u)^m.
\]
The proof uses derivative polynomials, a beta kernel, and a
coefficient-collapse argument.  We then establish strict sign
alternation for several coefficient families.  In the boundary
shifted-hyperbolic case the sign is controlled by generalized Bernoulli
polynomials, while the arctanh and logarithmic tangent families are
identified with continuous dual Hahn and continuous Hahn polynomials.
Their zero distributions turn the observed alternation into structural
statements valid for all admissible indices.

Finally, with $\lambda(s)=(1-2^{-s})\zeta(s)$, we show that the
parallel zeta--lambda and beta identities are the two binary
specializations of a single Lerch-transcendent scheme with
$\varepsilon\in\{0,1\}$.  The same convention
$\varepsilon=0$ for the beta case and $\varepsilon=1$ for the
zeta--lambda case is used throughout.  Besides the shifted hyperbolic,
odd-$\sinh$, logarithmic tangent, and $\tanh$ families, this
framework yields a further unified pair of reciprocal-arctanh integral
formulae on $(0,1)$.
\end{abstract}

\vspace{0.2cm}

\paragraph{Notation.}
We write \(\mathbb N=\{1,2,\ldots\}\).  The Riemann zeta, Dirichlet beta,
and Dirichlet lambda functions are denoted by \(\zeta(s)\), \(\beta(s)\),
and \(\lambda(s)\), respectively, with
\[
\beta(s)=\sum_{j=0}^{\infty}\frac{(-1)^j}{(2j+1)^s},
\qquad
\lambda(s)=\sum_{j=0}^{\infty}\frac{1}{(2j+1)^s}
          =(1-2^{-s})\zeta(s).
\]
For the Dirichlet lambda function, see \cite{HuKim2019}.
The Lerch transcendent is denoted by
\(\Phi(q,s,a)=\sum_{j=0}^{\infty}q^j/(j+a)^s\) in its domain of
convergence; see \cite[Eq.~(25.14.1)]{NIST:DLMF}.  We use
\((a)_r\) for the rising factorial.  The Chebyshev polynomial of the second kind
\cite{MasonHandscomb2003,NIST:DLMF} is denoted by \(U_m\) and normalized by
\[
U_m(\cos\theta)=\frac{\sin((m+1)\theta)}{\sin\theta}.
\]
For a power series or Laurent series \(F\), the notation \([x^r]F(x)\)
denotes the coefficient of \(x^r\) in the expansion at the origin.
Generalized Bernoulli polynomials and the continuous Hahn and continuous
dual Hahn polynomials are used with the normalizations cited at their
first occurrence.  In the Lerch-unification section,
\(\varepsilon\in\{0,1\}\) is used consistently as a binary parameter,
with \(\varepsilon=0\) selecting the beta-type formula and
\(\varepsilon=1\) the zeta--lambda-type formula.

\vspace{0.5cm}

\section*{Introduction}

Integrals built from powers of hyperbolic functions, singular factors
\(x^{-n}\), and logarithmic tangent kernels often reduce to finite linear
combinations of special values of the Riemann zeta and Dirichlet beta
functions.  Recursive coefficient descriptions for related unshifted
hyperbolic kernels were obtained by Kyrion \cite{Kyrion2025}, while
contour representations underlying shifted hyperbolic and logarithmic
tangent families appear in \cite{talla_waffo_integral_2025}.  Related
logarithmic and hyperbolic-tangent integrals have also been studied by
contour methods in \cite{ReynoldsStauffer2020}.

\vspace{0.2cm}

A coefficient-extraction viewpoint was developed in
\cite{TallaWaffo2026}: coefficients in several zeta- and beta-expansions
were encoded by a single extraction involving powers of \(\arcsin\),
Chebyshev polynomials, or the elementary kernel \(u\cot u\).  In
particular, that paper established the zeta-type identity
\[
\int_0^\infty\frac{\tanh^{m+1}x}{x^{n+1}}\,dx
=
(-1)^{(m-n)/2}
\sum_{p=\lceil n/2\rceil}^{(m+n)/2}
\binom{2p}{n}(2^{2p+1}-1)
\frac{\zeta(2p+1)}{\pi^{2p}}
[u^{m+n-2p}](u\cot u)^{m+1}.
\]
The first purpose of the present paper is to complete this picture on the
beta side.  We prove the companion formula
\[
\int_0^\infty
\frac{\tanh^m x}{x^n\cosh x}\,dx
=
(-1)^{(m-n)/2}
\sum_{p=\lceil n/2\rceil}^{(m+n)/2}
2^{2p}\binom{2p-1}{n-1}
\frac{\beta(2p)}{\pi^{2p-1}}
[u^{m+n-2p}]
\frac{u}{\sin u}(u\cot u)^m,
\]
for \(m\geq n\geq1\) and \(m+n\) even.  The additional factor
\(u/\sin u\) is the natural beta counterpart of the cotangent kernel;
it emerges from a beta-type hyperbolic integral after repeated
integration by parts and a coefficient collapse.

\vspace{0.2cm}

The second purpose is structural.  Several coefficient families that
initially appear unrelated exhibit strict sign alternation.  For the
boundary shifted-hyperbolic family we reduce the relevant coefficients
to generalized Bernoulli polynomials.  For the arctanh and logarithmic
tangent families, the coefficient-generating polynomials are identified
with continuous dual Hahn and continuous Hahn polynomials.  Positivity
of the corresponding orthogonality measures and the location of their
zeros then force the required sign patterns.

\vspace{0.2cm}

The third purpose is to expose the common zeta--beta mechanism.  The
Dirichlet lambda function satisfies
\(2^s\lambda(s)=(2^s-1)\zeta(s)\), while at \(a=\frac12\) the Lerch
transcendent contains both beta and lambda values:
\[
\Phi\!\left(-1,s,\frac12\right)=2^s\beta(s),
\qquad
\Phi\!\left(1,s,\frac12\right)=2^s\lambda(s).
\]
Consequently, the beta and zeta--lambda identities can be written as two
specializations of one binary Lerch formula.  We use the convention
\[
\varepsilon=0\quad\text{for the beta case},
\qquad
\varepsilon=1\quad\text{for the zeta--lambda case}
\]
throughout the unification.  In addition to compressing the shifted
hyperbolic, odd-\(\sinh\), logarithmic tangent, and \(\tanh\) pairs, this
viewpoint leads to a unified pair of reciprocal-arctanh integrals.

\vspace{0.2cm}

For clarity about the dependence on \cite{TallaWaffo2026}, the
Chebyshev--arcsine coefficient formulae for the shifted hyperbolic,
odd-\(\sinh\), and logarithmic tangent families, together with the
zeta-type \(\tanh\) formula, are taken from that paper.  The new results
proved here are the beta-type \(\tanh\) formula, the sign-alternation
theorems, the systematic Lerch compression, and the reciprocal-arctanh
pair derived at the end of the paper.

\vspace{0.2cm}

The paper is organized as follows.  \Cref{sec:coefficient-formulae}
proves the beta-type \(\tanh\) formula and records explicit examples.
\Cref{sec:sign_alternation} establishes strict alternation in three
coefficient families.  \Cref{sec:lerch-unification} develops the binary
Lerch formulation and derives the reciprocal-arctanh consequence. 

\vspace{0.3cm}

\section{Coefficient Formula for $\displaystyle\int_0^\infty
\frac{\tanh^m x}{x^n\cosh x}\,dx$}\label[section]{sec:coefficient-formulae}

The main result of this section is the beta-type companion of the
cotangent coefficient formula recalled in the introduction.  The proof
is arranged so that the origin of each factor remains visible.  We first
encode derivatives of \(\tanh^m x/\cosh x\) by a polynomial recurrence,
then insert the beta kernel inherited from \cite{TallaWaffo2026}, and
finally collapse the resulting finite coefficient sum to the single
kernel \(\dfrac{u}{\sin u}(u\cot u)^m\).  The two examples at the end
illustrate how the general formula produces explicit finite combinations
of even Dirichlet beta values.

\begin{lemma}[Derivative-polynomial structure]
\label[lemma]{lem:sech-tanh-derivative-polynomial}
Let \(m\geq1\) be an integer and define
\(f_m(x):=\tanh^m x/\cosh x\).
For every integer \(r\geq0\), there exists a polynomial \(P_{m,r}\)
such that
\[
f_m^{(r)}(x)
=
\frac{1}{\cosh x}P_{m,r}(\tanh x).
\]
The polynomials satisfy \(P_{m,0}(t)=t^m\) and
\[
P_{m,r+1}(t)
=
(1-t^2)P_{m,r}'(t)-tP_{m,r}(t).
\]
If \(0\leq r\leq m\), then the smallest and largest powers occurring
in \(P_{m,r}\) are \(t^{m-r}\) and \(t^{m+r}\), respectively,
and only powers of the same parity occur. In particular, if \(m+r\)
is odd, then
\[
P_{m,r}(t)
=
\sum_{k=(m-r-1)/2}^{(m+r-1)/2}
a_{m,r,k}\,t^{2k+1}
\]
for suitable coefficients \(a_{m,r,k}\).
\end{lemma}

\begin{proof}
Put \(t=\tanh x\). Since \(dt/dx=1-t^2\) and
\(\dfrac{d}{dx}(\cosh^{-1}x)=-t\cosh^{-1}x\), we obtain
\[
\frac{d}{dx}
\left(\frac{P(t)}{\cosh x}\right)
=
\frac{(1-t^2)P'(t)-tP(t)}{\cosh x},
\]
which proves the recurrence.

Starting from \(P_{m,0}(t)=t^m\), induction shows that
\(P_{m,r}\) contains only powers
\(t^{m-r},t^{m-r+2},\ldots,t^{m+r}\).
At each step the coefficient of the lowest power is multiplied by
\(m-r\), while that of the highest power is multiplied by
\(-(m+r+1)\). Thus both endpoint coefficients are nonzero for
\(0\leq r\leq m\).
\end{proof}

\begin{lemma}[Beta kernel]
\label[lemma]{lem:beta-kernel}
For every integer \(k\geq0\),
\[
\int_0^\infty
\frac{\sinh^{2k+1}x}{x\cosh^{2k+2}x}\,dx
=
\sum_{p=1}^{k+1}
2^{2p}\frac{\beta(2p)}{\pi^{2p-1}}
[z^{2k+1}]
(z^2-1)^k(\arcsin z)^{2p-1},
\]
where
\(\beta(s):=\sum_{j=0}^{\infty}(-1)^j/(2j+1)^s\)
is the Dirichlet beta function.
\end{lemma}

\begin{proof}
This is the special case \(n=k+1\) of Corollary~1.3 in
\cite{TallaWaffo2026}.
\end{proof}

\begin{lemma}[Coefficient collapse]
\label[lemma]{lem:beta-coefficient-collapse}
Let \(m\geq n\geq1\), with \(m+n\) even, and put
\(k_0:=(m-n)/2\) and \(K:=(m+n)/2-1\).
By \cref{lem:sech-tanh-derivative-polynomial}, write
\[
P_{m,n-1}(t)
=
\sum_{k=k_0}^{K}a_{m,n,k}t^{2k+1},
\qquad
A_{m,n}(w)
:=
\sum_{k=k_0}^{K}a_{m,n,k}w^k.
\]
For \(p\geq1\), define
\[
S_p
:=
\sum_{k=k_0}^{K}
a_{m,n,k}
[z^{2k+1}]
(z^2-1)^k(\arcsin z)^{2p-1}.
\]
Then
\[
S_p
=
\begin{cases}
0, & 2p<n,\\[2mm]
\displaystyle
(-1)^{(m-n)/2}
\frac{(2p-1)!}{(2p-n)!}
[u^{m+n-2p}]
\frac{u}{\sin u}(u\cot u)^m,
& 2p\geq n.
\end{cases}
\]
\end{lemma}

\begin{proof}
Since \((z^2-1)^k=z^{2k}(1-z^{-2})^k\), we have
\[
[z^{2k+1}]
(z^2-1)^k(\arcsin z)^{2p-1}
=
[z]\,(1-z^{-2})^k(\arcsin z)^{2p-1}.
\]
Hence
\[
S_p
=
[z]\,
A_{m,n}(1-z^{-2})(\arcsin z)^{2p-1}.
\]
By Cauchy's coefficient formula,
\[
S_p
=
\frac{1}{2\pi i}
\oint
A_{m,n}(1-z^{-2})
\frac{(\arcsin z)^{2p-1}}{z^2}\,dz,
\]
where the contour is a sufficiently small positively oriented circle
around the origin.

Set \(u=\arcsin z\), so that \(z=\sin u\) and
\(dz=\cos u\,du\). Since \(1-\sin^{-2}u=-\cot^2u\),
\[
S_p
=
\frac{1}{2\pi i}
\oint
u^{2p-1}A_{m,n}(-\cot^2u)
\frac{\cos u}{\sin^2u}\,du.
\]
From \(P_{m,n-1}(t)=tA_{m,n}(t^2)\), it follows that
\[
A_{m,n}(-\cot^2u)
=
\frac{P_{m,n-1}(i\cot u)}{i\cot u},
\]
and therefore
\[
S_p
=
\frac{1}{2\pi i}
\oint
u^{2p-1}
\frac{P_{m,n-1}(i\cot u)}{i\sin u}\,du.
\]

Now set \(x=i(\pi/2-u)\). Then
\(\tanh x=i\cot u\), \(\cosh x=\sin u\), and
\(d/dx=i\,d/du\). Moreover,
\[
\frac{\tanh^m x}{\cosh x}
=
i^m\frac{\cos^m u}{\sin^{m+1}u}.
\]
Using the defining identity for \(P_{m,n-1}\), we obtain
\[
\frac{P_{m,n-1}(i\cot u)}{\sin u}
=
i^{m+n-1}
\frac{d^{n-1}}{du^{n-1}}
\left(
\frac{\cos^m u}{\sin^{m+1}u}
\right).
\]

Since \(m+n\) is even,
\[
S_p
=
\frac{(-1)^{(m+n-2)/2}}{2\pi i}
\oint
u^{2p-1}
\frac{d^{\,n-1}}{du^{\,n-1}}
\left(
\frac{\cos^m u}{\sin^{m+1}u}
\right)\,du.
\]

Set \(F(u)=\dfrac{\cos^m u}{\sin^{m+1}u}\). Then
\[
S_p
=
\frac{(-1)^{(m+n-2)/2}}{2\pi i}
\oint u^{2p-1}F^{(n-1)}(u)\,du.
\]

On a closed contour, integration by parts gives
\[
\oint f(u)g'(u)\,du=-\oint f'(u)g(u)\,du,
\]
because \(\displaystyle\oint (f(u)g(u))'\,du=0\). This applies here since the contour
does not pass through any pole of the meromorphic functions involved.

Applying this identity \(n-1\) times yields
\[
\oint u^{2p-1}F^{(n-1)}(u)\,du
=
(-1)^{n-1}
\oint
\frac{d^{\,n-1}}{du^{\,n-1}}
\bigl(u^{2p-1}\bigr)
F(u)\,du.
\]

Since \(u^{2p-1}\) has degree \(2p-1\), its \((n-1)\)-st derivative
vanishes whenever \(n-1>2p-1\), i.e. whenever \(2p<n\). Hence
\[
\boxed{S_p=0 \qquad \text{for } 2p<n.}
\]

If \(2p\ge n\), then
\[
\frac{d^{\,n-1}}{du^{\,n-1}}u^{2p-1}
=
\frac{(2p-1)!}{(2p-n)!}\,u^{2p-n},
\]
and therefore
\[
S_p
=
\frac{(-1)^{(m+n-2)/2+n-1}}{2\pi i}
\frac{(2p-1)!}{(2p-n)!}
\oint
u^{2p-n}
\frac{\cos^m u}{\sin^{m+1}u}\,du.
\]

Since
\[
(-1)^{(m+n-2)/2+n-1}
=
(-1)^{(m-n)/2}
\]
and
\[
\frac{\cos^m u}{\sin^{m+1}u}
=
u^{-m-1}
\frac{u}{\sin u}(u\cot u)^m,
\]
we obtain
\[
S_p
=
(-1)^{(m-n)/2}
\frac{(2p-1)!}{(2p-n)!}
\frac{1}{2\pi i}
\oint
u^{\,2p-n-m-1}
\frac{u}{\sin u}(u\cot u)^m\,du.
\]

By Cauchy's coefficient formula,
\[
\boxed{
S_p
=
(-1)^{(m-n)/2}
\frac{(2p-1)!}{(2p-n)!}
[u^{m+n-2p}]
\frac{u}{\sin u}(u\cot u)^m
}
\qquad (2p\ge n).
\]

\end{proof}

\begin{proposition}
\label[proposition]{prop:sech-tanh-integral}
Let \(m,n\) be positive integers satisfying
\(m\geq n\geq1\) and \(m+n\equiv0\pmod2\). Then
\[
\boxed{
\int_0^\infty
\frac{\tanh^m x}{x^n\cosh x}\,dx=
(-1)^{(m-n)/2}
\sum_{p=\lceil n/2\rceil}^{(m+n)/2}
2^{2p}
\binom{2p-1}{n-1}
\frac{\beta(2p)}{\pi^{2p-1}}
[u^{m+n-2p}]
\frac{u}{\sin u}(u\cot u)^m.
}
\]
\end{proposition}

\begin{proof}
Set \(f_m(x):=\tanh^m x/\cosh x\). Near the origin,
\(f_m(x)=O(x^m)\), while \(f_m\) and all its derivatives decay
exponentially as \(x\to\infty\). Since \(m\geq n\), all boundary
terms below vanish. Thus \(n-1\) integrations by parts give
\[
\int_0^\infty \frac{f_m(x)}{x^n}\,dx
=
\frac1{(n-1)!}
\int_0^\infty \frac{f_m^{(n-1)}(x)}{x}\,dx.
\]

By \cref{lem:sech-tanh-derivative-polynomial},
\[
f_m^{(n-1)}(x)
=
\frac{P_{m,n-1}(\tanh x)}{\cosh x},
\qquad
P_{m,n-1}(t)
=
\sum_{k=k_0}^{K}a_{m,n,k}t^{2k+1},
\]
where \(k_0=(m-n)/2\) and \(K=(m+n)/2-1\). Hence
\[
\begin{aligned}
\int_0^\infty
\frac{\tanh^m x}{x^n\cosh x}\,dx
&=
\frac1{(n-1)!}
\sum_{k=k_0}^{K}a_{m,n,k}
\int_0^\infty
\frac{\tanh^{2k+1}x}{x\cosh x}\,dx
\\
&=
\frac1{(n-1)!}
\sum_{k=k_0}^{K}a_{m,n,k}
\int_0^\infty
\frac{\sinh^{2k+1}x}{x\cosh^{2k+2}x}\,dx.
\end{aligned}
\]

Applying \cref{lem:beta-kernel} termwise yields
\[
\frac1{(n-1)!}
\sum_{k=k_0}^{K}a_{m,n,k}
\sum_{p=1}^{k+1}
2^{2p}\frac{\beta(2p)}{\pi^{2p-1}}
[z^{2k+1}]
(z^2-1)^k(\arcsin z)^{2p-1}.
\]
Interchanging the finite sums gives
\[
\frac1{(n-1)!}
\sum_{p=1}^{K+1}
2^{2p}\frac{\beta(2p)}{\pi^{2p-1}}
\sum_{k=\max(k_0,p-1)}^{K}
a_{m,n,k}
[z^{2k+1}]
(z^2-1)^k(\arcsin z)^{2p-1}.
\]
For \(k<p-1\), the coefficient in the inner sum vanishes since
\((\arcsin z)^{2p-1}\) starts in degree \(2p-1>2k+1\).
Thus the lower limit may be replaced by \(k_0\).

By \cref{lem:beta-coefficient-collapse}, the inner sum vanishes for
\(2p<n\), so \(p\geq\lceil n/2\rceil\). For the remaining terms,
\[
\begin{aligned}
&\sum_{k=k_0}^{K}
a_{m,n,k}
[z^{2k+1}]
(z^2-1)^k(\arcsin z)^{2p-1}
\\
&\qquad=
(-1)^{(m-n)/2}
\frac{(2p-1)!}{(2p-n)!}
[u^{m+n-2p}]
\frac{u}{\sin u}(u\cot u)^m.
\end{aligned}
\]
Finally,
\[
\frac1{(n-1)!}\frac{(2p-1)!}{(2p-n)!}
=
\binom{2p-1}{n-1},
\qquad
K+1=\frac{m+n}{2},
\]
which proves the result.
\end{proof}

\begin{example}
\label{ex:m8-sech-tanh}
For \(m=8\), the coefficient kernel is
\[
\frac{u}{\sin u}(u\cot u)^8
=
1-\frac52u^2+\frac{301}{120}u^4
-\frac{1247}{1008}u^6+\frac{35}{128}u^8+O(u^{10}).
\]
Thus, for \(n\in\{2,4,6,8\}\), \cref{prop:sech-tanh-integral} yields
\[
\int_0^\infty
\frac{\tanh^8x}{x^2\cosh x}\,dx
=
-\frac{35}{32}\frac{\beta(2)}{\pi}
+\frac{1247}{21}\frac{\beta(4)}{\pi^3}
-\frac{2408}{3}\frac{\beta(6)}{\pi^5}
+4480\frac{\beta(8)}{\pi^7}
-9216\frac{\beta(10)}{\pi^9}.
\]
\[
\int_0^\infty
\frac{\tanh^8x}{x^4\cosh x}\,dx
=
\frac{35}{8}\frac{\beta(4)}{\pi^3}
-\frac{49880}{63}\frac{\beta(6)}{\pi^5}
+\frac{67424}{3}\frac{\beta(8)}{\pi^7}
-215040\frac{\beta(10)}{\pi^9}
+675840\frac{\beta(12)}{\pi^{11}}.
\]
\[
\int_0^\infty
\frac{\tanh^8x}{x^6\cosh x}\,dx
=
-\frac{35}{2}\frac{\beta(6)}{\pi^5}
+\frac{19952}{3}\frac{\beta(8)}{\pi^7}
-\frac{1618176}{5}\frac{\beta(10)}{\pi^9}
+4730880\frac{\beta(12)}{\pi^{11}}
-21086208\frac{\beta(14)}{\pi^{13}}.
\]
\[
\int_0^\infty
\frac{\tanh^8x}{x^8\cosh x}\,dx
=
70\frac{\beta(8)}{\pi^7}
-\frac{319232}{7}\frac{\beta(10)}{\pi^9}
+3390464\frac{\beta(12)}{\pi^{11}}
-70287360\frac{\beta(14)}{\pi^{13}}
+421724160\frac{\beta(16)}{\pi^{15}}.
\]
\end{example}

\begin{example}
\label{ex:m9-sech-tanh}
For \(m=9\), the coefficient kernel is
\[
\frac{u}{\sin u}(u\cot u)^9
=
1-\frac{17}{6}u^2+\frac{239}{72}u^4
-\frac{30539}{15120}u^6
+\frac{25609}{40320}u^8+O(u^{10}).
\]
Thus, for \(n\in\{1,3,5,7,9\}\), \cref{prop:sech-tanh-integral} yields
\[
\int_0^\infty
\frac{\tanh^9x}{x\cosh x}\,dx
=
\frac{25609}{10080}\frac{\beta(2)}{\pi}
-\frac{30539}{945}\frac{\beta(4)}{\pi^3}
+\frac{1912}{9}\frac{\beta(6)}{\pi^5}
-\frac{2176}{3}\frac{\beta(8)}{\pi^7}
+1024\frac{\beta(10)}{\pi^9}.
\]
\[
\int_0^\infty \frac{\tanh^9x}{x^3\cosh x}\,dx
=
-\frac{25609}{840}\frac{\beta(4)}{\pi^3}
+\frac{244312}{189}\frac{\beta(6)}{\pi^5}
-\frac{53536}{3}\frac{\beta(8)}{\pi^7}
+104448\frac{\beta(10)}{\pi^9}
-225280\frac{\beta(12)}{\pi^{11}}.
\]
\[
\int_0^\infty \frac{\tanh^9x}{x^5\cosh x}\,dx
=
\frac{25609}{126}\frac{\beta(6)}{\pi^5}
-\frac{488624}{27}\frac{\beta(8)}{\pi^7}
+428288\frac{\beta(10)}{\pi^9}
-3829760\frac{\beta(12)}{\pi^{11}}
+11714560\frac{\beta(14)}{\pi^{13}}.
\]
\[
\int_0^\infty \frac{\tanh^9x}{x^7\cosh x}\,dx
=
-\frac{51218}{45}\frac{\beta(8)}{\pi^7}
+\frac{7817984}{45}\frac{\beta(10)}{\pi^9}
-\frac{18844672}{3}\frac{\beta(12)}{\pi^{11}}
+79659008\frac{\beta(14)}{\pi^{13}}
-328007680\frac{\beta(16)}{\pi^{15}}.
\]
\[
\int_0^\infty \frac{\tanh^9x}{x^9\cosh x}\,dx
=
\frac{204872}{35}\frac{\beta(10)}{\pi^9}
-\frac{85997824}{63}\frac{\beta(12)}{\pi^{11}}
+69994496\frac{\beta(14)}{\pi^{13}}
-1194885120\frac{\beta(16)}{\pi^{15}}
+6372720640\frac{\beta(18)}{\pi^{17}}.
\]
\end{example}

\section{Sign alternation in coefficient formulae}\label[section]{sec:sign_alternation}

The explicit formulae above reveal an alternating pattern that is not a
numerical accident.  This section isolates that phenomenon and explains
three mechanisms behind it.  For the boundary shifted-hyperbolic family,
a change of variables reduces the relevant coefficients to generalized
Bernoulli polynomials.  For the arctanh and logarithmic tangent families,
the coefficient-generating polynomials are identified with continuous
dual Hahn or continuous Hahn polynomials.  Orthogonality then forces
their zeros onto positive or symmetric real sets, from which strict
coefficient alternation follows.

\begin{proposition}[Alternation in the boundary case \(k=n-1\)]
\label[proposition]{prop:alternating-boundary}
Let \(n\geq1\) and set \(k=n-1\). Then
\[
\int_0^\infty
\frac{\sinh((2n-1)x)}{x\cosh^{2n+1}x}\,dx
=
\sum_{p=1}^{n}
(2^{2p+1}-1)
[z^{2n}]\,U_{2n-2}(z)(\arcsin z)^{2p}
\frac{\zeta(2p+1)}{\pi^{2p}},
\]
and
\[
\int_0^\infty
\frac{\sinh((2n-1)x)}{x\cosh^{2n}x}\,dx
=
\sum_{p=1}^{n}
2^{2p}
[z^{2n-1}]\,U_{2n-2}(z)(\arcsin z)^{2p-1}
\frac{\beta(2p)}{\pi^{2p-1}}.
\]
Moreover, for every \(1\leq p\leq n\),
\[
\begin{cases}
\displaystyle
(-1)^{p-1}
[z^{2n}]\,U_{2n-2}(z)(\arcsin z)^{2p}>0,\\[2mm]
\displaystyle
(-1)^{p-1}
[z^{2n-1}]\,U_{2n-2}(z)(\arcsin z)^{2p-1}>0.
\end{cases}
\]
Hence the summands on the right-hand sides of both integral
representations alternate strictly in sign, beginning with a positive
term.
\end{proposition}

\begin{proof}
The two integral identities are the specialization \(k=n-1\) of
\cite[Prop.~1.2]{TallaWaffo2026}.  It remains only to establish the
asserted sign pattern.

The Chebyshev polynomials of the second kind satisfy
\[
U_j(\cos\theta)
=
\frac{\sin((j+1)\theta)}{\sin\theta};
\]
see \cite[Eq.~(18.5.2)]{NIST:DLMF}. Hence
\[
U_{2n-2}(\sin t)
=
(-1)^{n-1}
\frac{\cos((2n-1)t)}{\cos t}.
\tag{1}
\]

For \(N,r\geq0\), set
\[
C_{N,r}
:=
[z^N]\,U_{2n-2}(z)(\arcsin z)^r.
\]
Since \(z=\sin t\) is locally invertible at the origin, Cauchy's
coefficient formula and \((1)\) give
\[
C_{N,r}
=
(-1)^{n-1}
[t^{N-r}]
\cos((2n-1)t)
\left(\frac{t}{\sin t}\right)^{N+1}.
\tag{2}
\]

We now use the generalized Bernoulli polynomials \(B_q^{(\alpha)}(x)\),
defined by
\[
\left(\frac{w}{e^w-1}\right)^\alpha e^{xw}
=
\sum_{q=0}^{\infty}
B_q^{(\alpha)}(x)\frac{w^q}{q!};
\]
see \cite[Eq.~(24.16.1)]{NIST:DLMF}. Put
\(\nu=2n-1\). Since
\[
\frac{t}{\sin t}
=
\frac{2it\,e^{it}}{e^{2it}-1},
\]
taking \(w=2it\) gives
\[
e^{\pm i\nu t}
\left(\frac{t}{\sin t}\right)^\alpha
=
\left(\frac{w}{e^w-1}\right)^\alpha
e^{(\alpha\pm\nu)w/2}.
\]
Averaging the two identities and using
\[
B_q^{(\alpha)}(\alpha-x)
=
(-1)^q B_q^{(\alpha)}(x),
\]
which follows directly from the generating function, yields
\[
[t^{2j}]
\cos((2n-1)t)
\left(\frac{t}{\sin t}\right)^\alpha
=
\frac{(-4)^j}{(2j)!}
B_{2j}^{(\alpha)}
\left(
\frac{\alpha-(2n-1)}{2}
\right).
\tag{3}
\]

Fix \(1\leq p\leq n\) and put \(j=n-p\). The two coefficient families
are obtained by choosing
\[
(N,r,\alpha,\xi)
=
\begin{cases}
(2n,\,2p,\,2n+1,\,1),
& \text{for the zeta expansion},\\[2mm]
(2n-1,\,2p-1,\,2n,\,\frac12),
& \text{for the beta expansion}.
\end{cases}
\tag{4}
\]
In both cases \(N-r=2j\), \(\alpha=N+1\), and
\(\xi=(\alpha-(2n-1))/2\). Thus \((2)\) and \((3)\) give
\[
C_{N,r}
=
(-1)^{n-1}
\frac{(-4)^j}{(2j)!}
B_{2j}^{(\alpha)}(\xi).
\tag{5}
\]

We use the classical factorization
\[
B_{m-1}^{(m)}(x)
=
\prod_{\ell=1}^{m-1}(x-\ell),
\qquad m\geq2,
\tag{6}
\]
see \cite[p.~147]{Norlund1924} or
\cite[Eq.~(28)]{DilcherVignat2017}.
Moreover, differentiation of the generating function gives
\[
\frac{d^q}{dx^q}B_s^{(\alpha)}(x)
=
\frac{s!}{(s-q)!}B_{s-q}^{(\alpha)}(x).
\tag{7}
\]
Since \(\alpha-1-r=2j\), equations \((6)\) and \((7)\) imply
\[
B_{2j}^{(\alpha)}(\xi)
=
\frac{(2j)!}{(\alpha-1)!}
\left.
\frac{d^r}{dx^r}
\prod_{\ell=1}^{\alpha-1}(x-\ell)
\right|_{x=\xi}.
\tag{8}
\]

By the Leibniz rule, the derivative in \((8)\) is a positive linear
combination of products of exactly \(2j\) undifferentiated linear
factors. In the zeta case, evaluation at \(x=1\) annihilates every term
in which the factor \(x-1\) has not been differentiated; every surviving
term therefore contains \(2j\) factors \(1-\ell<0\). In the beta case,
every remaining factor is of the form \(\frac12-\ell<0\), again with
exactly \(2j\) such factors. Since \(2j\) is even, all surviving products
are positive, and hence
\[
B_{2j}^{(\alpha)}(\xi)>0
\]
in both cases.

It follows from \((5)\) that
\[
\operatorname{sgn}(C_{N,r})
=
(-1)^{n-1+j}
=
(-1)^{p-1},
\]
since \(j=n-p\). Therefore
\[
\begin{cases}
\displaystyle
\operatorname{sgn}
\left(
[z^{2n}]\,U_{2n-2}(z)(\arcsin z)^{2p}
\right)
=
(-1)^{p-1},
\\[3mm]
\displaystyle
\operatorname{sgn}
\left(
[z^{2n-1}]\,U_{2n-2}(z)(\arcsin z)^{2p-1}
\right)
=
(-1)^{p-1}.
\end{cases}
\]
Finally, all remaining factors in the two sums are positive for
\(p\geq1\). Hence both right-hand sides have the strict sign pattern
\[
+,-,+,-,\ldots,
\]
which proves the result.
\end{proof}

\begin{proposition}[Strict alternation in the arctanh coefficient formulae]
\label[proposition]{prop:arctanh-alternation}
Let \(n\geq1\), and define
\[
a_{p,n} := [z^{2n}](z^2-1)^{n-1}(\arcsin z)^{2p}, \qquad b_{p,n} := [z^{2n-1}](z^2-1)^{n-1}(\arcsin z)^{2p-1}.
\]
Then
\[
\int_0^1 \frac{x^{2n-1}}{\operatorname{arctanh}x}\,dx = \sum_{p=1}^{n} (2^{2p+1}-1)a_{p,n} \frac{\zeta(2p+1)}{\pi^{2p}},
\]
and
\[
\int_0^1 \frac{x^{2n-1}} {\sqrt{1-x^2}\,\operatorname{arctanh}x}\,dx = \sum_{p=1}^{n} 2^{2p}b_{p,n} \frac{\beta(2p)}{\pi^{2p-1}}.
\]
Moreover, for every \(1\leq p\leq n\),
\[
\begin{cases} (-1)^{p-1}a_{p,n}>0,\\[1mm] (-1)^{p-1}b_{p,n}>0. \end{cases}
\]
Consequently, the summands on the right-hand sides of both integral representations
alternate strictly in sign, beginning with a positive term.
\end{proposition}
\begin{proof} With the substitution \(x=\tanh t\), the two integral identities are exactly the case \(k=n-1\) of \cite[Cor.~1.3]{TallaWaffo2026}.  It therefore remains to establish the sign pattern. Introduce
\[
P_n(x) := \sum_{p=1}^{n}\frac{a_{p,n}}{(2p)!}x^{2p}, \qquad Q_n(x) := \sum_{p=1}^{n}\frac{b_{p,n}}{(2p-1)!}x^{2p-1}.
\]
By the definitions of \(a_{p,n}\) and \(b_{p,n}\),
\[
\begin{cases} \displaystyle P_n(x) = [z^{2n}](z^2-1)^{n-1}\cosh(x\arcsin z),\\[2mm] \displaystyle Q_n(x) = [z^{2n-1}](z^2-1)^{n-1}\sinh(x\arcsin z).
\end{cases}
\tag{1}
\]
Using the hypergeometric expansions
\[
\cosh(x\arcsin z) = {}_2F_1\!\left( -\frac{ix}{2},\frac{ix}{2}; \frac12;z^2 \right)
\]
and
\[
\sinh(x\arcsin z) = xz\, {}_2F_1\!\left( \frac{1-ix}{2},\frac{1+ix}{2}; \frac32;z^2 \right),
\]
coefficient extraction in \((1)\) gives
\[
\begin{cases} \displaystyle P_n(x) = \frac{x^2}{2}\, {}_3F_2\!\left( \begin{matrix} 1-n,\ 1-\frac{ix}{2},\ 1+\frac{ix}{2}\\ \frac32,\ 2 \end{matrix};1 \right), \\[5mm] \displaystyle Q_n(x) = x\, {}_3F_2\!\left( \begin{matrix} 1-n,\ \frac{1-ix}{2},\ \frac{1+ix}{2}\\ \frac32,\ 1 \end{matrix};1 \right).
\end{cases}
\tag{2}
\]
Recall that the continuous dual Hahn polynomials are defined by
\[
S_m(y^2;a,b,c) = (a+b)_m(a+c)_m {}_3F_2\!\left( \begin{matrix} -m,\ a+iy,\ a-iy\\ a+b,\ a+c \end{matrix};1 \right);
\]
see \cite[Eq.~(18.26.2)]{NIST:DLMF} and \cite[\S9.3]{KoekoekLeskySwarttouw2010}. Comparing
this representation with \((2)\) yields
\[
\begin{cases} \displaystyle P_n(x) = \frac{4^{\,n-1}x^2}{(2n)!}\, S_{n-1}\!\left( \frac{x^2}{4};1,\frac12,1 \right), \\[4mm] \displaystyle Q_n(x) = \frac{4^{\,n-1}x}{(2n-1)!}\, S_{n-1}\!\left( \frac{x^2}{4};\frac12,1,\frac12 \right).
\end{cases}
\tag{3}
\]
All parameters occurring in \((3)\) are positive. Hence the two continuous dual Hahn
families are orthogonal with respect to a positive measure on \((0,\infty)\); see
\cite[Table~18.25.1 and Eqs.~(18.25.6)--(18.25.8)]{NIST:DLMF} or
\cite[\S9.3]{KoekoekLeskySwarttouw2010}. By the standard zero theorem for orthogonal
polynomials, all zeros are simple and lie in the interval of orthogonality
\cite[\S18.2(vi)]{NIST:DLMF}. For \(n\geq2\), let \(\xi_{1,n},\ldots,\xi_{n-1,n}>0\) and
\(\eta_{1,n},\ldots,\eta_{n-1,n}>0\) denote the zeros of the two continuous dual Hahn
polynomials in \((3)\), respectively. Their leading coefficient is \((-1)^{n-1}\)
\cite[Table~18.25.2]{NIST:DLMF}; therefore
\[
\begin{cases} \displaystyle S_{n-1}\!\left(X;1,\frac12,1\right) = \prod_{j=1}^{n-1}(\xi_{j,n}-X), \\[3mm] \displaystyle S_{n-1}\!\left(X;\frac12,1,\frac12\right) = \prod_{j=1}^{n-1}(\eta_{j,n}-X).
\end{cases}
\tag{4}
\]
Since all \(\xi_{j,n}\) and \(\eta_{j,n}\) are positive, the coefficients of each product in
\((4)\), ordered by increasing powers of \(X\), have the strict sign pattern
\[
+,-,+,-,\ldots.
\]
Substituting \(X=x^2/4\) into \((3)\) therefore shows that
\[
\begin{cases} \displaystyle \operatorname{sgn}[x^{2p}]P_n(x)=(-1)^{p-1},\\[2mm] \displaystyle \operatorname{sgn}[x^{2p-1}]Q_n(x)=(-1)^{p-1}. \end{cases}
\]
Since
\[
[x^{2p}]P_n(x)=\frac{a_{p,n}}{(2p)!}, \qquad [x^{2p-1}]Q_n(x)=\frac{b_{p,n}}{(2p-1)!},
\]
we obtain
\[
\begin{cases} (-1)^{p-1}a_{p,n}>0,\\[1mm] (-1)^{p-1}b_{p,n}>0. \end{cases}
\]
The case \(n=1\) is immediate from \(P_1(x)=x^2/2\) and \(Q_1(x)=x\). Finally, all remaining
factors in the two sums are positive: \(2^{2p+1}-1>0\), \(\zeta(2p+1)>0\), \(\beta(2p)>0\),
and \(\pi>0\). Hence both sums alternate strictly as
\[
+,-,+,-,\ldots,
\]
as claimed.
\end{proof}

\begin{proposition}[Strict alternation in the logarithmic tangent formulae]
\label[proposition]{prop:logtan-alternation}
Let \(n\geq1\), and define
\[
a_{p,n}:=[x]\,U_{2n-1}\!\left(\frac1x\right)(\arcsin x)^{2p}, \qquad b_{p,n}:=[x]\,U_{2n-2}\!\left(\frac1x\right)(\arcsin x)^{2p-1}.
\]
Then
\[
\int_0^{\pi/4}\frac{\sin(4nx)}{\ln(\tan x)}\,dx = \frac{(-1)^n}{2} \sum_{p=1}^{n} (2^{2p+1}-1)a_{p,n} \frac{\zeta(2p+1)}{\pi^{2p}},
\]
and
\[
\int_0^{\pi/4}\frac{\cos((4n-2)x)}{\ln(\tan x)}\,dx = (-1)^n \sum_{p=1}^{n} 2^{2p-1}b_{p,n} \frac{\beta(2p)}{\pi^{2p-1}}.
\]
Moreover, for \(1\leq p\leq n\),
\[
\begin{cases} (-1)^{n-p}a_{p,n}>0,\\[1mm] (-1)^{n-p}b_{p,n}>0. \end{cases}
\]
Consequently, after multiplication by the common factor \((-1)^n\), the summands in both
representations alternate strictly as
\[
-,+,-,+,\ldots.
\]
\end{proposition}
\begin{proof}
The two integral identities are \cite[Prop.~2.3]{TallaWaffo2026}.  It therefore remains to
determine the signs of \(a_{p,n}\) and \(b_{p,n}\). Define
\[
P_n(t):=\sum_{p=1}^{n}\frac{a_{p,n}}{(2p)!}t^{2p}, \qquad Q_n(t):=\sum_{p=1}^{n}\frac{b_{p,n}}{(2p-1)!}t^{2p-1}.
\]
By construction,
\[
\begin{cases} \displaystyle P_n(t)=[x]\,U_{2n-1}\!\left(\frac1x\right)\cosh(t\arcsin x),\\[2mm] \displaystyle Q_n(t)=[x]\,U_{2n-2}\!\left(\frac1x\right)\sinh(t\arcsin x).
\end{cases}
\tag{1}
\]
From the explicit Chebyshev expansion
\[
U_m(y)= \sum_{\ell=0}^{\lfloor m/2\rfloor} (-1)^\ell \frac{(m-\ell)!}{\ell!(m-2\ell)!}(2y)^{m-2\ell},
\]
see \cite[Eq.~(18.5.11\_3)]{NIST:DLMF}, one obtains
\[
\begin{cases} \displaystyle U_{2n-1}\!\left(\frac1x\right) = \sum_{q=1}^{n} (-1)^{n-q} \frac{2^{2q-1}(n+q-1)!}{(n-q)!(2q-1)!}\, x^{-(2q-1)},\\[4mm] \displaystyle U_{2n-2}\!\left(\frac1x\right) = \sum_{q=0}^{n-1} (-1)^{n-1-q} \frac{2^{2q}(n+q-1)!}{(n-1-q)!(2q)!}\, x^{-2q}.
\end{cases}
\tag{2}
\]
The standard hypergeometric identities \cite[Eqs.~(15.4.12),(15.4.16)]{NIST:DLMF} give
\[
\cosh(t\arcsin x) = {}_2F_1\!\left( -\frac{it}{2},\frac{it}{2};\frac12;x^2 \right), \qquad \sinh(t\arcsin x) = tx\,{}_2F_1\!\left( \frac{1-it}{2},\frac{1+it}{2};\frac32;x^2 \right).
\]
Hence, for \(q\geq1\),
\[
[x^{2q}]\cosh(t\arcsin x) = \frac{t^2}{(2q)!} \prod_{j=1}^{q-1}(t^2+4j^2),
\]
while, for \(q\geq0\),
\[
[x^{2q+1}]\sinh(t\arcsin x) = \frac{t}{(2q+1)!} \prod_{j=1}^{q}\bigl(t^2+(2j-1)^2\bigr). \tag{3}
\]
Substitution of \((2)\) and \((3)\) into \((1)\) gives
\[
\begin{aligned} P_n(t) &= \sum_{q=1}^{n} (-1)^{n-q} \frac{2^{2q-1}(n+q-1)!} {(n-q)!(2q-1)!(2q)!}\, t^2\prod_{j=1}^{q-1}(t^2+4j^2),\\ Q_n(t) &= \sum_{q=0}^{n-1} (-1)^{n-1-q} \frac{2^{2q}(n+q-1)!} {(n-1-q)!(2q)!(2q+1)!}\, t\prod_{j=1}^{q}\bigl(t^2+(2j-1)^2\bigr).
\end{aligned}
\tag{4}
\]
Writing the products in Pochhammer notation and simplifying the finite sums yields
\[
\begin{cases} \displaystyle P_n(t) = (-1)^{n-1}nt^2\, {}_4F_3\!\left( \begin{matrix} 1-n,\ n+1,\ 1-\frac{it}{2},\ 1+\frac{it}{2}\\ \frac32,\frac32,2 \end{matrix};1 \right),\\[5mm] \displaystyle Q_n(t) = (-1)^{n-1}t\, {}_4F_3\!\left( \begin{matrix} 1-n,\ n,\ \frac{1-it}{2},\ \frac{1+it}{2}\\ \frac12,1,\frac32 \end{matrix};1 \right).
\end{cases}
\tag{5}
\]
The corresponding terminating transformations give
\[
\begin{cases} \displaystyle P_n(t) = (-1)^{n-1} \frac{2^{2n-1}t^2}{(2n)!} \left(\frac32\right)_{n-1}^{\!2} {}_3F_2\!\left( \begin{matrix} 1-n,\ \frac12+it,\ \frac12-it\\ \frac32,\frac32 \end{matrix};1 \right),\\[5mm] \displaystyle Q_n(t) = (-1)^{n-1}t\, {}_3F_2\!\left( \begin{matrix} 2-2n,\ 2n,\ \frac{1+it}{2}\\ 1,\frac32 \end{matrix};1 \right).
\end{cases}
\tag{6}
\]
Recall that the continuous dual Hahn polynomials satisfy
\[
S_m(y^2;a,b,c) = (a+b)_m(a+c)_m {}_3F_2\!\left( \begin{matrix} -m,\ a+iy,\ a-iy\\ a+b,\ a+c \end{matrix};1 \right),
\]
see \cite[Eq.~(18.26.2)]{NIST:DLMF}, whereas the continuous Hahn polynomials admit the
representation
\[
p_m(y;a,b,\bar a,\bar b) = \frac{i^m(a+\bar a)_m(a+\bar b)_m}{m!} {}_3F_2\!\left( \begin{matrix} -m,\ m+2\operatorname{Re}(a+b)-1,\ a+iy\\ a+\bar a,\ a+\bar b \end{matrix};1 \right),
\]
see \cite[Eq.~(18.20.9)]{NIST:DLMF}. Comparing with \((6)\) gives
\[
\begin{cases} \displaystyle P_n(t) = (-1)^{n-1} \frac{2^{2n-1}t^2}{(2n)!}\, S_{n-1}\!\left(t^2;\frac12,1,1\right),\\[4mm] \displaystyle Q_n(t) = \frac{2^{4n-4}(2n-2)!}{(4n-3)!}\, t\, p_{2n-2}\!\left( \frac t2;\frac12,1,\frac12,1 \right).
\end{cases}
\tag{7}
\]
For \(S_{n-1}(X;\frac12,1,1)\), all parameters are positive, so this continuous dual Hahn
family is orthogonal with respect to a positive measure on \(X>0\); see \cite[Table~18.25.1
and Eqs.~(18.25.6)--(18.25.8)]{NIST:DLMF}. Hence its zeros are simple and positive by the
standard zero theorem for orthogonal polynomials \cite[\S18.2(vi)]{NIST:DLMF}. Since its
leading coefficient is \((-1)^{n-1}\) \cite[Table~18.25.2]{NIST:DLMF}, there exist
\(\xi_{1,n},\ldots,\xi_{n-1,n}>0\) such that
\[
S_{n-1}\!\left(X;\frac12,1,1\right) = \prod_{j=1}^{n-1}(\xi_{j,n}-X). \tag{8}
\]
Thus its coefficients, ordered by increasing powers of \(X\), alternate strictly in sign.
For the continuous Hahn polynomial in \((7)\), the parameters
\(a=c=\frac12\) and \(b=d=1\) give a positive even orthogonality
weight on \(\mathbb R\); see
\cite[Eqs.~(18.19.1)--(18.19.3)]{NIST:DLMF}.  Hence its
\(2n-2\) zeros are simple and real.  Because the weight is even,
uniqueness of the orthogonal polynomial of each degree implies the
parity relation
\[
p_r(-x;\tfrac12,1,\tfrac12,1)
=
(-1)^r p_r(x;\tfrac12,1,\tfrac12,1).
\]
Thus the zeros of the even polynomial \(p_{2n-2}\) occur in pairs
\(\pm\rho_{j,n}\), with \(\rho_{j,n}>0\).  Moreover, its leading
coefficient is positive in the normalization
\cite[Eq.~(18.20.9)]{NIST:DLMF}.  Consequently,
\[
p_{2n-2}\!\left(
\frac t2;\frac12,1,\frac12,1
\right)
=
c_n\prod_{j=1}^{n-1}(t^2-\rho_{j,n}^2),
\qquad c_n>0.
\tag{9}
\]
Its coefficients as a polynomial in \(t^2\) therefore alternate
strictly. Since the prefactors in \((7)\) are nonzero and have fixed sign, comparison with
the definitions of \(P_n\) and \(Q_n\) gives, for \(1\leq p\leq n\),
\[
\begin{cases} \operatorname{sgn}(a_{p,n})=(-1)^{n-p},\\[1mm] \operatorname{sgn}(b_{p,n})=(-1)^{n-p}. \end{cases}
\]
Finally, all zeta-, beta-, and powers-of-\(\pi\) factors in the two integral formulae are
positive. Thus the sign of the \(p\)-th summand, after inclusion of the common factor
\((-1)^n\), is \((-1)^n(-1)^{n-p}=(-1)^p\). Hence both sums alternate strictly as
\[
-,+,-,+,\ldots,
\]
as claimed.
\end{proof}

\section{Binary unification through the Lerch transcendent}\label[section]{sec:lerch-unification}

The zeta- and beta-type formulae encountered above occur in parallel
pairs: the parity of the hyperbolic exponent, the coefficient index, the
power of \(\arcsin\), and the special value all change together.  The
Dirichlet lambda function makes this symmetry transparent.  For
\(\Re(s)>1\),
\[
\lambda(s)=\sum_{j=0}^{\infty}\frac{1}{(2j+1)^s}
          =(1-2^{-s})\zeta(s),
\]
see \cite[Eq.~(1.9)]{HuKim2019}, and the defining series of the Lerch
transcendent \cite[Eq.~(25.14.1)]{NIST:DLMF} gives
\[
\Phi\!\left(-1,s,\frac12\right)=2^s\beta(s),
\qquad
\Phi\!\left(1,s,\frac12\right)
=2^s\lambda(s)=(2^s-1)\zeta(s).
\tag{1}
\]
Thus, for the binary parameter \(\varepsilon\in\{0,1\}\),
\[
\Phi\!\left(2\varepsilon-1,\,2p+\varepsilon,\,\frac12\right)
=
\begin{cases}
2^{2p}\beta(2p),&\varepsilon=0,\\[1mm]
(2^{2p+1}-1)\zeta(2p+1),&\varepsilon=1.
\end{cases}
\tag{2}
\]
The identity \((2)\) is the common bridge used below.  The first five
propositions in this section are compact reformulations or immediate
consequences of coefficient formulae already available in
\cite{TallaWaffo2026} or proved in \Cref{sec:coefficient-formulae};
they are recorded to make the common binary structure explicit.  The reciprocal-arctanh result at the end is
derived separately.

\subsection{Binary forms of the coefficient identities}

\begin{proposition}[Unified Lerch form]
\label[proposition]{prop:unified-lerch-shifted}
Let \(n\geq1\), \(0\leq k<n\), and let
\(\varepsilon\in\{0,1\}\). Then
\[
\boxed{
\int_0^\infty
\frac{\sinh((2k+1)x)}
{x\cosh^{\,2n+\varepsilon}x}\,dx
=
\sum_{p=1}^{n}
\frac{
\Phi\!\left(
2\varepsilon-1,\,
2p+\varepsilon,\,
\frac12
\right)}
{\pi^{\,2p-1+\varepsilon}}
[z^{\,2n-1+\varepsilon}]
U_{2k}(z)(\arcsin z)^{\,2p-1+\varepsilon}
}
\]
where \(\Phi\) denotes the Lerch transcendent.  The choices
\(\varepsilon=0\) and \(\varepsilon=1\) recover, respectively, the
Dirichlet beta and zeta--lambda coefficient formulae.
\end{proposition}

\begin{proof}
By \cite[Prop.~1.2]{TallaWaffo2026},
\[
\begin{cases}
\displaystyle
\int_0^\infty
\frac{\sinh((2k+1)x)}{x\cosh^{2n}x}\,dx
=
\sum_{p=1}^{n}
2^{2p}
[z^{2n-1}]U_{2k}(z)(\arcsin z)^{2p-1}
\frac{\beta(2p)}{\pi^{2p-1}},
\\[4mm]
\displaystyle
\int_0^\infty
\frac{\sinh((2k+1)x)}{x\cosh^{2n+1}x}\,dx
=
\sum_{p=1}^{n}
(2^{2p+1}-1)
[z^{2n}]U_{2k}(z)(\arcsin z)^{2p}
\frac{\zeta(2p+1)}{\pi^{2p}}.
\end{cases}
\]
Using \((2)\), the two cases are obtained by setting
\(\varepsilon=0\) and \(\varepsilon=1\), respectively; the exponents
\(2n+\varepsilon\), \(2n-1+\varepsilon\), and
\(2p-1+\varepsilon\) simultaneously reproduce the denominator power,
coefficient index, and power of \(\arcsin z\).
\end{proof}

\begin{proposition}[Unified Lerch form for the even-frequency shifted integrals]
\label[proposition]{prop:unified-lerch-even}
Let \(m,k\in\mathbb N\) with \(m\geq k\geq 1\), and let
\(\varepsilon\in\{0,1\}\). Then
\[
\int_0^\infty
\frac{\sinh(2kx)}
{x\cosh^{\,2m+1+\varepsilon}x}\,dx
=
\sum_{p=1}^{m}
\frac{
\Phi\!\left(
2\varepsilon-1,\,
2p+\varepsilon,\,
\frac12
\right)}
{\pi^{\,2p-1+\varepsilon}}
[z^{\,2m+\varepsilon}]
U_{2k-1}(z)(\arcsin z)^{\,2p-1+\varepsilon}.
\]
\end{proposition}

\begin{proof}
The result follows directly from \cref{prop:unified-lerch-shifted}.
Indeed,
\[
2\cosh x\,\sinh(2kx)
=
\sinh((2k+1)x)+\sinh((2k-1)x).
\]
Hence, with \(n=m+1\), the integral on the left is one half of the sum of
the two shifted integrals in \cref{prop:unified-lerch-shifted}
corresponding to the indices \(k\) and \(k-1\). The Chebyshev recurrence
\[
U_{2k}(z)+U_{2k-2}(z)=2z\,U_{2k-1}(z)
\]
then lowers the coefficient index by one and gives
\[
[z^{2m+\varepsilon}]
U_{2k-1}(z)(\arcsin z)^{2p-1+\varepsilon}.
\]
The term \(p=m+1\) inherited from \cref{prop:unified-lerch-shifted}
vanishes by degree and parity, so the sum terminates at \(p=m\). This proves
the stated formula. The condition \(m\geq k\geq 1\) ensures that the shifted indices \(k\) and \(k-1\) are both admissible.
\end{proof}

\begin{proposition}[Unified Lerch form for odd powers of \(\sinh\)]
\label[proposition]{prop:unified-lerch-sinh-powers}
Let \(n\geq1\), \(0\leq k<n\), and let
\(\varepsilon\in\{0,1\}\). Then
\[
\int_0^\infty
\frac{\sinh^{2k+1}x}
{x\cosh^{\,2n+\varepsilon}x}\,dx
=
\sum_{p=1}^{n}
\frac{
\Phi\!\left(
2\varepsilon-1,\,
2p+\varepsilon,\,
\frac12
\right)}
{\pi^{\,2p-1+\varepsilon}}
[z^{\,2n-1+\varepsilon}]
(z^2-1)^k
(\arcsin z)^{\,2p-1+\varepsilon}.
\]
\end{proposition}

\begin{proof}
By \cite[Cor.~1.3]{TallaWaffo2026}, the corresponding zeta and beta
coefficient formulae are
\[
\begin{cases}
\displaystyle
\int_0^\infty
\frac{\sinh^{2k+1}x}{x\cosh^{2n}x}\,dx
=
\sum_{p=1}^{n}
2^{2p}
[z^{2n-1}](z^2-1)^k(\arcsin z)^{2p-1}
\frac{\beta(2p)}{\pi^{2p-1}},
\\[4mm]
\displaystyle
\int_0^\infty
\frac{\sinh^{2k+1}x}{x\cosh^{2n+1}x}\,dx
=
\sum_{p=1}^{n}
(2^{2p+1}-1)
[z^{2n}](z^2-1)^k(\arcsin z)^{2p}
\frac{\zeta(2p+1)}{\pi^{2p}}.
\end{cases}
\]
Applying \((2)\) and reading the three exponents through
\(\varepsilon\) gives the asserted formula.
\end{proof}

\begin{proposition}[Unified Lerch form for the logarithmic tangent integrals]
\label[proposition]{prop:unified-lerch-logtan}
Let \(n\geq1\) and \(\varepsilon\in\{0,1\}\). Then
\[
\int_0^{\pi/4}
\frac{
\cos\!\left(
(4n-2+2\varepsilon)x-\frac{\pi\varepsilon}{2}
\right)}
{\ln(\tan x)}\,dx
=
\frac{(-1)^n}{2}
\sum_{p=1}^{n}
\frac{
\Phi\!\left(
2\varepsilon-1,\,
2p+\varepsilon,\,
\frac12
\right)}
{\pi^{\,2p-1+\varepsilon}}
[x]\,
U_{2n-2+\varepsilon}\!\left(\frac1x\right)
(\arcsin x)^{\,2p-1+\varepsilon}.
\]
\end{proposition}

\begin{proof}
By \cite[Prop.~3.3]{TallaWaffo2026},
\[
\begin{cases}
\displaystyle
\int_0^{\pi/4}
\frac{\cos((4n-2)x)}{\ln(\tan x)}\,dx
=
(-1)^n
\sum_{p=1}^{n}
2^{2p-1}
[x]\,U_{2n-2}\!\left(\frac1x\right)
(\arcsin x)^{2p-1}
\frac{\beta(2p)}{\pi^{2p-1}},
\\[4mm]
\displaystyle
\int_0^{\pi/4}
\frac{\sin(4nx)}{\ln(\tan x)}\,dx
=
\frac{(-1)^n}{2}
\sum_{p=1}^{n}
(2^{2p+1}-1)
[x]\,U_{2n-1}\!\left(\frac1x\right)
(\arcsin x)^{2p}
\frac{\zeta(2p+1)}{\pi^{2p}}.
\end{cases}
\]
The numerator is unified by
\[
\cos\!\left(
(4n-2+2\varepsilon)x-\frac{\pi\varepsilon}{2}
\right)
=
\begin{cases}
\cos((4n-2)x),&\varepsilon=0,\\
\sin(4nx),&\varepsilon=1.
\end{cases}
\]
Together with \((2)\), the common factor \(1/2\) gives
\(2^{2p-1}\beta(2p)\) when \(\varepsilon=0\) and the zeta coefficient
when \(\varepsilon=1\), proving the formula.
\end{proof}

\begin{proposition}[Unified Lerch form for the \(\tanh\) coefficient formulae]
\label[proposition]{prop:unified-lerch-tanh}
Let \(m\geq n\geq1\) with \(m+n\) even, and let
\(\varepsilon\in\{0,1\}\). Then
\[
\int_0^\infty
\frac{\tanh^{m+\varepsilon}x}
{x^{n+\varepsilon}\cosh^{\,1-\varepsilon}x}\,dx
=
(-1)^{\frac{m-n}{2}}
\sum_{p=\lceil n/2\rceil}^{(m+n)/2}
\binom{2p-1+\varepsilon}{n-1+\varepsilon}
\frac{
\Phi\!\left(
2\varepsilon-1,\,
2p+\varepsilon,\,
\frac12
\right)}
{\pi^{\,2p-1+\varepsilon}}
[u^{m+n-2p}]
\left(\frac{u}{\sin u}\right)^{1-\varepsilon}
(u\cot u)^{m+\varepsilon}.
\]
\end{proposition}

\begin{proof}
For \(\varepsilon=0\), \cref{prop:sech-tanh-integral} gives
\[
\int_0^\infty
\frac{\tanh^m x}{x^n\cosh x}\,dx
=
(-1)^{(m-n)/2}
\sum_{p=\lceil n/2\rceil}^{(m+n)/2}
2^{2p}
\binom{2p-1}{n-1}
\frac{\beta(2p)}{\pi^{2p-1}}
[u^{m+n-2p}]
\frac{u}{\sin u}(u\cot u)^m.
\]
The zeta-type companion proved in \cite{TallaWaffo2026} is
\[
\int_0^\infty
\frac{\tanh^{m+1}x}{x^{n+1}}\,dx
=
(-1)^{(m-n)/2}
\sum_{p=\lceil n/2\rceil}^{(m+n)/2}
\binom{2p}{n}
(2^{2p+1}-1)
\frac{\zeta(2p+1)}{\pi^{2p}}
[u^{m+n-2p}](u\cot u)^{m+1}.
\]
Formula \((2)\) combines the special-value factors, while
\[
\binom{2p-1+\varepsilon}{n-1+\varepsilon}
=
\begin{cases}
\binom{2p-1}{n-1},&\varepsilon=0,\\
\binom{2p}{n},&\varepsilon=1,
\end{cases}
\]
and
\[
\left(\frac{u}{\sin u}\right)^{1-\varepsilon}
(u\cot u)^{m+\varepsilon}
=
\begin{cases}
\displaystyle\frac{u}{\sin u}(u\cot u)^m,&\varepsilon=0,\\[2mm]
(u\cot u)^{m+1},&\varepsilon=1.
\end{cases}
\]
The left-hand side is unified in the same way, which proves the result.
\end{proof}

\begin{proposition}[Unified Lerch form for logarithmic polylogarithm integrals]
Let \(n\geq 1\) and let \(\varepsilon\in\{0,1\}\). Then
\[
\int_0^1
\frac{
\Im\!\left(
i^\varepsilon
\operatorname{Li}_{-2n-\varepsilon}
\!\left(i^{\,1+\varepsilon}x^{\,1+\varepsilon}\right)
\right)
}{x}
\ln\ln\frac1x\,dx
=
(-1)^{n-1+\varepsilon}
\frac{(2n-1+\varepsilon)!}
{2^{\,1+(2n+1)\varepsilon}\pi^{\,2n-1+\varepsilon}}
\Phi\!\left(
2\varepsilon-1,
2n+\varepsilon,
\frac12
\right).
\]
The choices \(\varepsilon=0\) and \(\varepsilon=1\) recover, respectively,
the Dirichlet-beta and zeta--lambda identities.
\end{proposition}

\begin{proof}
By \cite{TallaWaffo2026}, one has
\[
\begin{cases}
\displaystyle
\int_0^1
\frac{\Im\!\left(\operatorname{Li}_{-2n}(ix)\right)}{x}
\ln\ln\frac1x\,dx
=
(-1)^{n-1}
2^{2n-1}(2n-1)!
\frac{\beta(2n)}{\pi^{2n-1}},
\\[4mm]
\displaystyle
\int_0^1
\frac{\operatorname{Li}_{-2n-1}(-x^2)}{x}
\ln\ln\frac1x\,dx
=
(-1)^n
\frac{(2^{2n+1}-1)(2n)!}{2^{2n+2}}
\frac{\zeta(2n+1)}{\pi^{2n}}.
\end{cases}
\]

The two polylogarithmic kernels are unified by
\[
\Im\!\left(
i^\varepsilon
\operatorname{Li}_{-2n-\varepsilon}
\!\left(i^{\,1+\varepsilon}x^{\,1+\varepsilon}\right)
\right)
=
\begin{cases}
\displaystyle
\Im\!\left(\operatorname{Li}_{-2n}(ix)\right),
& \varepsilon=0,
\\[2mm]
\displaystyle
\operatorname{Li}_{-2n-1}(-x^2),
& \varepsilon=1.
\end{cases}
\]
Indeed, the first case is immediate, while for \(\varepsilon=1\),
\[
i^{\,1+\varepsilon}x^{\,1+\varepsilon}
=i^2x^2=-x^2,
\]
and \(\operatorname{Li}_{-2n-1}(-x^2)\) is real for \(0<x<1\), so that
\[
\Im\!\left(
i\,\operatorname{Li}_{-2n-1}(-x^2)
\right)
=
\operatorname{Li}_{-2n-1}(-x^2).
\]

Moreover, the Lerch transcendent satisfies
\[
\Phi\!\left(
2\varepsilon-1,
2n+\varepsilon,
\frac12
\right)
=
\begin{cases}
\displaystyle
2^{2n}\beta(2n),
& \varepsilon=0,
\\[2mm]
\displaystyle
(2^{2n+1}-1)\zeta(2n+1),
& \varepsilon=1.
\end{cases}
\]
Hence, for \(\varepsilon=0\),
\[
\frac{(-1)^{n-1}(2n-1)!}
{2\pi^{2n-1}}
\Phi\!\left(-1,2n,\frac12\right)
=
(-1)^{n-1}
2^{2n-1}(2n-1)!
\frac{\beta(2n)}{\pi^{2n-1}},
\]
whereas for \(\varepsilon=1\),
\[
\frac{(-1)^n(2n)!}
{2^{2n+2}\pi^{2n}}
\Phi\!\left(1,2n+1,\frac12\right)
=
(-1)^n
\frac{(2^{2n+1}-1)(2n)!}{2^{2n+2}}
\frac{\zeta(2n+1)}{\pi^{2n}}.
\]
Thus the two identities are precisely the binary specializations
\(\varepsilon=0\) and \(\varepsilon=1\) of the asserted formula.
\end{proof}

\subsection{Reciprocal arctanh integrals}

The preceding propositions reorganize known coefficient identities.  We
now derive a further pair directly from the two \(\tanh\)-coefficient
formulae and then place it in the same Lerch scheme.

\begin{proposition}[Unified Lerch form for reciprocal arctanh integrals]
\label[proposition]{prop:arctanh-lerch-unification}
Let \(m,n\) be positive integers satisfying
\[
m\ge n\ge 1,
\qquad
m+n\equiv 0 \pmod 2,
\]
and let \(\varepsilon\in\{0,1\}\). Then
\[
\begin{aligned}
\int_0^1
\frac{x^m}
{(1-x^2)^{(1-\varepsilon)/2}\operatorname{arctanh}^n x}\,dx
&=
(-1)^{(m-n)/2}
\sum_{p=\lceil n/2\rceil}^{(m+n)/2}
\binom{2p-1+\varepsilon}{n-1}
\\
&\quad\times
\frac{
\Phi\!\left(
2\varepsilon-1,\,
2p+\varepsilon,\,
\frac12
\right)}
{\pi^{\,2p-1+\varepsilon}}
\\
&\quad\times
[z^{\,m+\varepsilon}]
(1-z^2)^{(m-1)/2}
(\arcsin z)^{\,2p-n+\varepsilon}.
\end{aligned}
\]
The choices \(\varepsilon=0\) and \(\varepsilon=1\) yield,
respectively, the Dirichlet-beta and zeta--lambda identities, in
agreement with the convention used throughout this section.
\end{proposition}

\begin{proof}
We derive the two cases separately and then combine them by means of
the Lerch transcendent.

\paragraph{The zeta-type identity.}
Set
\[
I_{m,n}
:=
\int_0^1
\frac{x^m}{\operatorname{arctanh}^n x}\,dx.
\]
With the substitution
\[
x=\tanh t,
\qquad
dx=\frac{dt}{\cosh^2 t},
\qquad
\operatorname{arctanh}x=t,
\]
we obtain
\[
I_{m,n}
=
\int_0^\infty
\frac{\tanh^m t}{t^n\cosh^2 t}\,dt.
\]
Since
\[
\frac{d}{dt}\tanh^{m+1}t
=
(m+1)\frac{\tanh^m t}{\cosh^2 t},
\]
integration by parts gives
\[
I_{m,n}
=
\frac{n}{m+1}
\int_0^\infty
\frac{\tanh^{m+1}t}{t^{n+1}}\,dt.
\]
Indeed, the boundary term vanishes at infinity, while at the origin
\[
t^{-n}\tanh^{m+1}t
=
O(t^{m+1-n})
\longrightarrow 0
\]
because $m\ge n$.

By the zeta-type tanh coefficient formula proved in
\cite[Prop.~3.4]{TallaWaffo2026},
\[
\begin{aligned}
\int_0^\infty
\frac{\tanh^{m+1}t}{t^{n+1}}\,dt
&=
(-1)^{(m-n)/2}
\sum_{p=\lceil n/2\rceil}^{(m+n)/2}
\binom{2p}{n}
\frac{(2^{2p+1}-1)\zeta(2p+1)}
{\pi^{2p}}
\\
&\qquad\times
[u^{m+n-2p}]
(u\cot u)^{m+1}.
\end{aligned}
\]
Hence
\[
\begin{aligned}
I_{m,n}
&=
\frac{n}{m+1}
(-1)^{(m-n)/2}
\sum_{p=\lceil n/2\rceil}^{(m+n)/2}
\binom{2p}{n}
\frac{(2^{2p+1}-1)\zeta(2p+1)}
{\pi^{2p}}
\\
&\qquad\times
[u^{m+n-2p}]
(u\cot u)^{m+1}.
\end{aligned}
\]

We now transform the coefficient. Put
\[
r:=m+n-2p,
\qquad
q:=2p-n.
\]
By Cauchy's coefficient formula,
\[
\begin{aligned}
[u^r](u\cot u)^{m+1}
&=
\frac{1}{2\pi i}
\oint
u^q\cot^{m+1}u\,du.
\end{aligned}
\]
Set
\[
z=\sin u,
\qquad
u=\arcsin z.
\]
Then
\[
du=\frac{dz}{\sqrt{1-z^2}},
\qquad
\cot u=\frac{\sqrt{1-z^2}}{z},
\]
and therefore
\[
[u^r](u\cot u)^{m+1}
=
\frac{1}{2\pi i}
\oint
\frac{
(\arcsin z)^q(1-z^2)^{m/2}
}{
z^{m+1}
}\,dz.
\]

Since
\[
d\!\left((\arcsin z)^{q+1}\right)
=
(q+1)
\frac{(\arcsin z)^q}{\sqrt{1-z^2}}\,dz,
\]
we may write
\[
\begin{aligned}
[u^r](u\cot u)^{m+1}
&=
\frac{1}{q+1}
\frac{1}{2\pi i}
\oint
\frac{(1-z^2)^{(m+1)/2}}{z^{m+1}}
\,d\!\left((\arcsin z)^{q+1}\right).
\end{aligned}
\]
Integration by parts on the closed contour yields
\[
\oint F\,dG=-\oint G\,dF.
\]
Moreover,
\[
\frac{d}{dz}
\left(
\frac{(1-z^2)^{(m+1)/2}}{z^{m+1}}
\right)
=
-(m+1)
\frac{(1-z^2)^{(m-1)/2}}{z^{m+2}}.
\]
Consequently,
\[
\begin{aligned}
[u^r](u\cot u)^{m+1}
&=
\frac{m+1}{q+1}
\frac{1}{2\pi i}
\oint
\frac{
(1-z^2)^{(m-1)/2}
(\arcsin z)^{q+1}
}{
z^{m+2}
}\,dz
\\
&=
\frac{m+1}{q+1}
[z^{m+1}]
(1-z^2)^{(m-1)/2}
(\arcsin z)^{q+1}.
\end{aligned}
\]
Since $q=2p-n$, this becomes
\[
\boxed{
[u^{m+n-2p}](u\cot u)^{m+1}
=
\frac{m+1}{2p-n+1}
[z^{m+1}]
(1-z^2)^{(m-1)/2}
(\arcsin z)^{2p-n+1}.
}
\]
Substitution into the preceding formula gives
\[
\begin{aligned}
I_{m,n}
&=
(-1)^{(m-n)/2}
\sum_{p=\lceil n/2\rceil}^{(m+n)/2}
\frac{n}{2p-n+1}
\binom{2p}{n}
\frac{(2^{2p+1}-1)\zeta(2p+1)}
{\pi^{2p}}
\\
&\qquad\times
[z^{m+1}]
(1-z^2)^{(m-1)/2}
(\arcsin z)^{2p-n+1}.
\end{aligned}
\]
Using
\[
\frac{n}{2p-n+1}\binom{2p}{n}
=
\binom{2p}{n-1},
\]
we obtain the zeta-type formula
\begin{equation}
\label{eq:arctanh-zeta-type}
\boxed{
\begin{aligned}
\int_0^1
\frac{x^m}{\operatorname{arctanh}^n x}\,dx
&=
(-1)^{(m-n)/2}
\sum_{p=\lceil n/2\rceil}^{(m+n)/2}
\binom{2p}{n-1}
\frac{(2^{2p+1}-1)\zeta(2p+1)}
{\pi^{2p}}
\\
&\qquad\times
[z^{m+1}]
(1-z^2)^{(m-1)/2}
(\arcsin z)^{2p-n+1}.
\end{aligned}
}
\end{equation}

\paragraph{The Dirichlet-beta identity.}
Now set
\[
J_{m,n}
:=
\int_0^1
\frac{x^m}
{\sqrt{1-x^2}\operatorname{arctanh}^n x}\,dx.
\]
Again using $x=\tanh t$, we have
\[
\sqrt{1-x^2}
=
\frac{1}{\cosh t},
\qquad
dx=\frac{dt}{\cosh^2 t},
\]
and hence
\[
J_{m,n}
=
\int_0^\infty
\frac{\tanh^m t}{t^n\cosh t}\,dt.
\]
By \cref{prop:sech-tanh-integral},
\[
\begin{aligned}
J_{m,n}
&=
(-1)^{(m-n)/2}
\sum_{p=\lceil n/2\rceil}^{(m+n)/2}
2^{2p}
\binom{2p-1}{n-1}
\frac{\beta(2p)}{\pi^{2p-1}}
\\
&\qquad\times
[u^{m+n-2p}]
\frac{u}{\sin u}(u\cot u)^m.
\end{aligned}
\]

As above, put
\[
r=m+n-2p,
\qquad
q=2p-n.
\]
Cauchy's coefficient formula gives
\[
\begin{aligned}
[u^r]
\frac{u}{\sin u}(u\cot u)^m
&=
\frac{1}{2\pi i}
\oint
u^q
\frac{\cos^m u}{\sin^{m+1}u}\,du.
\end{aligned}
\]
With $z=\sin u$ this becomes
\[
\begin{aligned}
[u^r]
\frac{u}{\sin u}(u\cot u)^m
&=
\frac{1}{2\pi i}
\oint
\frac{
(\arcsin z)^q
(1-z^2)^{(m-1)/2}
}{
z^{m+1}
}\,dz
\\
&=
[z^m]
(1-z^2)^{(m-1)/2}
(\arcsin z)^q.
\end{aligned}
\]
Thus
\begin{equation}
\label{eq:arctanh-beta-type}
\boxed{
\begin{aligned}
\int_0^1
\frac{x^m}
{\sqrt{1-x^2}\operatorname{arctanh}^n x}\,dx
&=
(-1)^{(m-n)/2}
\sum_{p=\lceil n/2\rceil}^{(m+n)/2}
\binom{2p-1}{n-1}
\frac{2^{2p}\beta(2p)}
{\pi^{2p-1}}
\\
&\qquad\times
[z^m]
(1-z^2)^{(m-1)/2}
(\arcsin z)^{2p-n}.
\end{aligned}
}
\end{equation}

\paragraph{Lerch unification.}
For the Lerch transcendent,
\[
\Phi\!\left(-1,s,\frac12\right)=2^s\beta(s),
\qquad
\Phi\!\left(1,s,\frac12\right)=(2^s-1)\zeta(s).
\]
Therefore, for \(\varepsilon\in\{0,1\}\),
\[
\Phi\!\left(
2\varepsilon-1,\,
2p+\varepsilon,\,
\frac12
\right)
=
\begin{cases}
2^{2p}\beta(2p),&\varepsilon=0,\\[1mm]
(2^{2p+1}-1)\zeta(2p+1),&\varepsilon=1.
\end{cases}
\]
For \(\varepsilon=0\), the remaining factors in the asserted formula
become
\[
\binom{2p-1}{n-1},\qquad
[z^m](1-z^2)^{(m-1)/2}(\arcsin z)^{2p-n},
\]
which reproduces \eqref{eq:arctanh-beta-type}.  For
\(\varepsilon=1\), they become
\[
\binom{2p}{n-1},\qquad
[z^{m+1}](1-z^2)^{(m-1)/2}(\arcsin z)^{2p-n+1},
\]
which reproduces \eqref{eq:arctanh-zeta-type}.  This proves the unified
formula with the same binary convention as in \((2)\).
\end{proof}

\section*{Acknowledgments}

The author acknowledges the use of an AI language model for language and presentation assistance and for exploratory numerical checks.

\printbibliography

\end{document}

%% file: preamble.tex
\usepackage[margin=0.8in]{geometry}
\usepackage{amsmath, amssymb, amsfonts, mathtools}
\usepackage{amsthm}
\usepackage{mathrsfs}
\usepackage{stmaryrd}
\usepackage{esint}
\usepackage{eqnalign}

\usepackage{graphicx}
\usepackage{enumitem}
\usepackage[most]{tcolorbox}
\usepackage{float}
\usepackage{listings}
\usepackage{array}
\usepackage{booktabs}
\usepackage{xcolor}
\usepackage{emptypage}
\usepackage{tensor}
\usepackage{tabularx}

\usepackage{fancyhdr}
\usepackage{hyperref}
\usepackage[nameinlink,noabbrev]{cleveref}

\usepackage{titlesec}
\usepackage[T1]{fontenc}
\usepackage[utf8]{inputenc}
\usepackage{lmodern}
\usepackage{titling}

\usepackage[backend=biber,style=numeric]{biblatex}
\providecommand{\HeadTitle}{} % <-- TITRE VARIABLE
\providecommand{\HeadTitleTwo}{} % <-- TITRE VARIABLE
\providecommand{\HeadAuthor}{Luc Ramsès TALLA WAFFO} % <-- FIXE

\titleformat{\section}[block]
  {\normalfont\large\bfseries\itshape\centering}
  {§\thesection.}
  {1em}
  {}

\titleformat{\subsection}
  {\normalfont\normalsize\bfseries}
  {\thesubsection}
  {1em}
  {}

\newtheorem{theorem}{Theorem}[section]
\newtheorem{lemma}[theorem]{Lemma}
\newtheorem{proposition}[theorem]{Proposition}

\newtheorem{example}[theorem]{Example}

\crefname{example}{example}{examples}
\Crefname{example}{Example}{Examples}

\crefname{corollary}{corollary}{corollaries}
\Crefname{corollary}{Corollary}{Corollaries}

\crefname{definition}{definition}{definitions}
\Crefname{definition}{Definition}{Definitions}

\crefname{remark}{remark}{remarks}
\Crefname{remark}{Remark}{Remarks}

\crefname{conjecture}{conjecture}{conjectures}
\Crefname{conjecture}{Conjecture}{Conjectures}

\crefname{lemma}{lemma}{lemmas}
\Crefname{lemma}{Lemma}{Lemmas}

\crefname{proposition}{proposition}{propositions}
\Crefname{proposition}{Proposition}{Propositions}

\crefname{theorem}{theorem}{theorems}
\Crefname{theorem}{Theorem}{Theorems}

\numberwithin{equation}{section}

%% file: font.tex
\usepackage{fontspec}

%% file: titlepage.tex
\thispagestyle{fancy}

\vspace{0.2cm}

\begin{center}
\Large{\HeadTitleTwo}
\end{center}

\hspace{3cm}

\begin{center}
Luc Ramsès TALLA WAFFO \\
Technische Universität Darmstadt\\
Karolinenplatz 5, 64289 Darmstadt, Germany\\
ramses.talla@stud.tu-darmstadt.de\\
\vspace{0.5cm}
\today
\end{center}